\documentclass[12pt,leqno,draft]{amsart}
\usepackage{amsmath,color}
\usepackage{dsfont}
\newcommand{\N}{{\mathds N}}
\newcommand{\R}{{\mathds R}}
\newcommand{\Z}{{\mathds Z}}
\newcommand{\Var}{{\rm Var}}
\newcommand{\Cov}{{\rm Cov}}

\newtheorem{theorem}{Theorem}
\newtheorem{lemma}{Lemma}[section]
\newtheorem{definition}[lemma]{Definition}
\newtheorem{example}[lemma]{Example}
\begin{document}
\bibliographystyle{plain}
\title[LIL for $U$-Statistics] 
{Law of the Iterated Logarithm for $U$-Statistics of Weakly Dependent
Observations}

\author[Herold Dehling]{Herold Dehling}
\author[Martin Wendler]{Martin Wendler}
\today
\address{
Fakult\"at f\"ur Mathematik, Ruhr-Universit\"at Bo\-chum, 
44780 Bochum, Germany}
\email{herold.dehling@ruhr-uni-bochum.de}
\email{martin.wendler@ruhr-uni-bochum.de}
\thanks{Research supported by the DFG Sonderforschungsbereich 823
{\em Statistik nichtlinearer dynamischer Prozesse}
and the Studienstiftung des deutschen Volkes}
\begin{abstract}
The law of the iterated logarithm for partial sums of weakly dependent processes was intensively studied by Walter Philipp in the late 1960s and 1970s. In this paper, we aim to extend these results to nondegenerate $U$-statistics of data that are strongly mixing or functionals of an absolutely regular process. 

\end{abstract}
\maketitle
\centerline{{\sc Dedicated to the memory of Professor Walter 
Philipp (1936--2006)}}

\section{Introduction}

Let $(T_n)_{n\geq 1}$ be a sequence of random variables. We say that
$(T_n)_{n\geq 1}$ satisfies the law of the iterated logarithm (LIL), if 
$\Var(T_n) > 0$ for almost all $n\geq 1$ and 
\begin{eqnarray*}
\limsup_{n\rightarrow \infty} 
 \frac{T_n}{\sqrt{2\, \Var(T_n) \log \log \Var(T_n)}} &=& 1,\\
\liminf_{n\rightarrow \infty} 
 \frac{T_n}{\sqrt{2\, \Var(T_n) \log \log \Var(T_n)}} &=& -1
\end{eqnarray*}
almost surely (a.s.). The LIL was originally established 
for partial sums of independent identically distributed
random variables by Khintchine  in 1927 \cite{khin}. Hartman \& Wintner \cite{hart} were able
to prove Khintchine's result under the optimal condition
that the random variables have mean zero and finite second moments. Together
with the law of large numbers and the central limit theorem, the LIL is considered as one of the three classical limit
theorems in probability theory.

In a series of papers, starting in 1967 (\cite{phil},\cite{phi5},\cite{phi3},\cite{phi6}), Walter Philipp investigated
the LIL for partial sums of weakly dependent
processes. Independently, Iosifescu (1968 \cite{iosi}) and Reznick (1968 \cite{rezn}) 
studied the same problem; Oodaira \& Yoshihara (1971 \cite{ooda})
weakened their conditions.
In \cite{phil} Walter Philipp studied the LIL 
for stationary processes with finite moments of all order satisfying some
multiple mixing condition. In his proof Walter Philipp established 
sharp bounds on the $(2p)$-th moments of partial sums and classical techniques
such as the Borel-Cantelli lemma and maximal inequalities. In 
\cite{phi3} Walter Philipp investigated the LIL for $\psi$-mixing 
processes with finite $4$-th moment. 
The proof is based on a meta-theorem, stating that
'the LIL holds for any process for which the Borel-Cantelli lemma, the 
central
limit theorem with a reasonably good remainder and a certain maximal inequality
are valid.' This observation provided a guiding principle for many of the 
early proofs of the LIL for dependent processes.

Walter Philipp's interest in dependent processes arose from specific 
applications to analysis and probabilistic number theory. In all of
his works, Walter Philipp had very concrete applications in mind to which
he could apply his theoretical results. In a joint paper
 with Stackelberg \cite{phi6}, Walter Philipp
established the LIL for
the denominator of the $n$-th approximand in the continued fraction expansion.
The relation to weakly dependent processes is provided by the fact that
the digits in the continued fraction expansion form a $\psi$-mixing 
sequence. In \cite{phi5}, Walter Philipp investigated dynamical systems
arising from expanding piecewise linear transformations  of the unit 
interval; the map $T(x)=2\, x\, [\mod 1]$ being a special example. 
These processes can be shown to have a representation as functionals 
of an absolutely regular process.

In \cite{phi5}, Walter Philipp considered the uniform LIL,
i.e. the LIL for the supremum of partial sums of $f(X_i)-E(f(X_1))$, 
where $f$ ranges over a  class of functions. As an example, Walter Philipp
could study the discrepancy of sequences arising from expanding piecewise
linear maps. This paper marked the beginning of Walter Philipp's interest
in the LIL for empirical processes and 
for Banach space valued processes. In \cite{phi4}, Walter Philipp proved
a Strassen-type functional LIL for the empirical process of data that
have a representation as a functional of a strongly mixing process. 
In a joint paper with Kaufman \cite{kauf}, Walter Philipp
studied uniform LIL for classes of Lipschitz functions, among others for
processes of the form $X_k=\{n_k\, \omega\}$, $\omega \in [0,1]$,
where $(n_k)_{k\geq 1}$ is a lacunary sequence. 
The study of the uniform LIL leads directly to Banach space valued
random variables. The first LIL for weakly dependent Banach space valued
processes was proved by Philipp \& Kuelbs \cite{kuel} in the case
of uniformly mixing processes. Specializing to the case of Hilbert
space valued random variables, Dehling \& Philipp \cite{deh4}
extended this to strongly mixing processes.

In the early 1970s, motivated by Strassen's proof of the functional
LIL, Walter Philipp realized that almost sure 
invariance principles were ideal tools for proofs of the LIL. 
In 1974, in an AMS memoir coauthored
with Stout \cite{phi2}, Walter Philipp established almost sure invariance principles
for a large class of weakly dependent processes, including functionals of
absolutely regular processes. Philipp \& Stout were among the first to
recognize the power of the  martingale approximation technique, invented 
in 1969 by Gordin \cite{gord}. Finally, in their seminal 1979 paper \cite{ber2}, 
Berkes \& Philipp invented a new technique for
proving almost sure invariance principles that can be used also 
for vector valued processes. The Berkes-Philipp approximation technique
has been  the basis of most work on invariance principles and the 
LIL in the following decades. For an excellent survey on invariance principles see Philipp \cite{phi7}.

Many other authors have considered the LIL for 
partial sums of weakly dependent processes. Berkes (1975 \cite{berk}) treats the LIL
for trigonometric functions, Dabrowski (1985 \cite{dab2}) establishes the LIL
for associated random variables, Dabrowski \& Dehling (1988 \cite{dabr}) extended this to
weakly associated random vectors. For partial sums of strongly mixing 
processes, the sharpest results presently available are due to Rio (1995 \cite{rio}).

In the present paper, we investigate the LIL for
bivariate $U$-statistics of weakly dependent data. Given a symmetric,
measurable function $h:\R^2\rightarrow \R$ and a stationary stochastic
process, we define the $U$-statistic with kernel $h$ by
\[
 U_n(h)=\frac{1}{\binom{n}{2}} \sum_{1\leq i<j \leq n} h(X_i,X_j).
\]
Thus, $U_n(h)$ is the arithmetic mean of the values $h(X_i,X_j)$, 
$1\leq i <j \leq n$, and in that sense $U$-statistics are generalized
means. 
Many sample statistics can be written as a $U$-statistic, at least
asymptotically, and thus $U$-statistics are very important in 
statistical theory. 
$U$-statistics have been introduced independently by Halmos (1946 \cite{halm}) and
Hoeffding (1948 \cite{hoef}), in the case of i.i.d. observations. Halmos observed that
$U_n(h)$ is an unbiased estimator of $Eh(X_1,X_2)$, and in fact the 
minimum variance unbiased estimator in nonparametric models. Hoeffding
showed that $U_n(h)$ is asymptotically normal. 

\begin{example} Let $h\left(x_1,x_2\right)=\left|x_1-x_2\right|.$ Then the corresponding $U$-statistic is
\begin{equation*}
U_n\left(h\right)=\frac{2}{n(n-1)}\sum_{1\leq i<j\leq n}\left|X_i-X_j\right|,
\end{equation*}
known as Gini's mean difference.
\end{example}

\begin{example}Let $h\left(x_{1},x_{2}\right)=\int_{0}^{1}\left(\mathds{1}_{\left\{x_{1}\leq t\right\}}-t\right)\left(\mathds{1}_{\left\{x_{2}\leq t\right\}}-t\right)dt$. This leads to the following $U$-statistic:

\begin{align*}
U_{n}\left(h\right)&=\frac{1}{\binom{n}{2}}\sum_{1\leq i<j\leq n} h\left(X_{i},X_{j}\right)\\ &=\frac{1}{n\left(n-1\right)}\left(\sum_{i=1}^{n}\sum_{j=1}^{n}\int_{0}^{1}\left(\mathds{1}_{\left\{X_{i}\leq t\right\}}-t\right)\left(\mathds{1}_{\left\{X_{j}\leq t\right\}}-t\right)dt-\sum_{i=1}^{n}h\left(X_{i},X_{i}\right)\right)\\
&=\frac{n}{n-1}\int_{0}^{1}\left(\hat{F}_{n}\left(t\right)-t\right)^{2}dt-\frac{1}{n\left(n-1\right)}\sum_{i=1}^{n}h\left(X_{i},X_{i}\right)\\
&:=\frac{n}{n-1}V_{n}-\frac{1}{n\left(n-1\right)}\sum_{i=1}^{n}h\left(X_{i},X_{i}\right)
\end{align*}

$V_{n}$ is called Cramer-von Mises-Statistik and can be used for testing the hypothesis that $X_n$ has a uniform distribution on $\left[0,1\right]$ as an alternative to the Kolmogorow-Smirnoff-statistic $K_n:=\sup_{t\in\left[0,1\right]}|\hat{F}_{n}\left(t\right)-t|$ (also called discrepancy).
\end{example}

\begin{example} Let be $t\in\R$ and $h\left(x_1,x_2\right)=\mathds{1}_{\left\{\frac{1}{2}\left(x_1+x_2\right)\leq t\right\}}$. This kernel is related to the Hodges-Lehmann-estimator
\begin{equation*}
 H_n=\operatorname{median}\left\{\frac{X_i+X_j}{2}\big|1\leq i<j\leq n\right\},
\end{equation*}
as we will see later.
\end{example}

The key tool in the analysis of $U$-statistics is the Hoeffding decomposition,
introduced originally by Hoeffding (1948),
\[
 U_n(h)=\theta + \frac{2}{n} \sum_{i=1}^n h_1(X_i) 
 +U_n(h_2)
\]
Here, $\theta$, $h_1(x)$ and $h_2(x,y)$ are defined by
\begin{eqnarray*}
\theta &:=& Eh(X,Y) \\
h_1(x) &:=& Eh(x,Y) -\theta \\
h_2(x,y) &:=& h(x,y) - h_1(x) -h_1(y) -\theta,
\end{eqnarray*}
where $X, Y$ are independent random variables with the same distribution as
$X_1$. The linear term in the Hoeffding decomposition,
$\frac{2}{n} \sum_{i=1}^n h_1(X_i)$, can be treated by standard limit 
theorems for partial sum processes. Note that, by definition, 
$h_1(X_i)$ are centered (i.e. mean zero) random variables. The kernel $h_2(x,y)$ has the
property that for every $x\in \R$
\[
 Eh_2(x,Y)=0;
\]
kernels with this property are called degenerate. It turns out that 
$U_n(h_2)$ is generally stochastically dominated by the linear term, and 
thus as a result the asymptotic behavior of $U_n(h)$ is 
the same as that of $\frac{2}{n} \sum_{i=1}^n h_1(X_i)$. Depending on the
type of limit theorem and the conditions imposed on the process
$(X_i)_{i\geq 1}$, this can be more or less difficult to establish. 

For degenerate $U$-statistics of i.i.d. observations, Dehling, Denker and Philipp (1985 \cite{deh3}) and Dehling (1989 \cite{deh2}) 
established the LIL. They could show that
\[
 \limsup_{n\rightarrow \infty}
  \frac{1}{n\, \log\log n} \sum_{1\leq i<j \leq n}
 h_2(X_i,X_j) = c_{h_2},
\]
where $c_{h_2}$ is the largest eigenvalue of the integral operator with
kernel $h_2$. This was extended to mixing random variables by Kanagawa and Yoshihara \cite{kana} under the condition that the eigenvalues of $h_2$ decreas quickly, that is hard to verify in practice.

Recall that strong mixing coefficients of a stationary stochastic process
$(X_n)_{n\in\N}$ are defined by
\[
 \alpha (k) := \sup \left\{\left| P (A\cap B) - P (A) P (B) \right| : A \in \mathcal{F}^n_1, B \in \mathcal{F}^\infty_{n+k}, n \in \N \right\}
 \]
where $\mathcal{F}^l_a$ denotes the $\sigma-$field generated by the random variables $X_a, \ldots, X_l.$. For a detailed description of the various mixing conditions see Doukhan \cite{douk} and Bradley \cite{bra2}. The absolute regularity coefficients are defined as
\begin{equation*}
\beta (k) := \sup_{n\in\N}E \sup \{ \left| P (A / \mathcal{F}_{-\infty}^n) - P (A) \right| : A \in \mathcal {F}^\infty_{n + k}\},
\end{equation*}
 We say that $\left(X_n\right)_{n\in\mathds{N}}$ is strongly mixing if $\lim_{k \rightarrow \infty} \alpha (k) = 0$ and absolutely regular if $\lim_{k \rightarrow \infty} \beta (k) = 0.$ Absolute regularity is a stronger assumption than strong mixing, as $\alpha\left(k\right)\leq\beta\left(k\right)$.

We will consider strongly mixing sequences and functionals of absolutely regular sequences. Let $ \left(Z_n\right)_{n\in\mathds{Z}}$ be a stationary sequence of random variables satisfying the absolute regularity condition $\beta\left(k\right)\rightarrow 0$ as $k\rightarrow\infty$. We call a sequence $\left(X_n\right)_{n\in\mathds{N}}$ a one-sided functional of $\left(Z_n\right)_{n\in\mathds{N}}$ if there is a measurable function $f:\R^\N\rightarrow\R$ such that
\[
X_n = f (\left( Z_{n + k}\right)_{k \geq 0}).
\]
In addition we will assume that $\left(X_n\right)_{n\in\mathds{N}}$ satisfies the $r$-approximation condition:

\begin{definition} Let be $r\geq1$. We say that $\left(X_n\right)_{n\in\mathds{Z}}$  satisfies the $r$-approximating condition with constants $\left(a_n\right)_{n\in\N}$ if
\[
\left\| X_1 - E (X_1 / \mathcal {F}^l_{0}) \right\|_r \leq a_l \qquad l = 0, 1,2 \ldots
\]
where $\lim_{l \rightarrow  \infty} a_l = 0$ and $ \mathcal {F}_{0}^l$ is the $\sigma-$ field generated by $Z_{0}, \ldots, Z_l$ and $\left\|Y\right\|_r=\left(E\left|Y\right|^r\right)^{\frac{1}{r}}$.
\end{definition}

\begin{example} Let be $\left(Z_n\right)_{n\in\N}$ be independent with $P\left[X_n=1\right]=P\left[X_n=0\right]=\frac{1}{2}$ and
\begin{equation*}
 X_n=\sum_{k=n}^{\infty}\frac{1}{2^{k-n+1}}Z_k.
\end{equation*}
Note that $\left(X_n\right)_{n\in\N}$ is a deterministic sequence, as $X_{n+1}=T\left(X_n\right):=2X_n\ \mod\ 1$. Thus $\left(X_n\right)_{n\in\N}$ is not strongly mixing, but nevertheless this sequence satisfies the $r$-approximating condition for every $r\geq 1$, as
\begin{equation*}
\left\| X_1 - E (X_1 / \mathcal {F}^l_{0}) \right\|_r=\left\|\sum_{k=l+1}^{\infty}\frac{1}{2^{k+1}}Z_k\right\|_r\leq\sum_{k=l+1}^{\infty}\frac{1}{2^{k+1}}=\frac{1}{2^l}=:a_l.
\end{equation*}
\end{example}

$U$-statistic have not only been studied for i.i.d. data, but also under various mixing conditions. While under independence, the summands of $U_n\left(h_2\right)$ are uncorrelated, they can be correlated if the random variables $\left(X_n\right)_{n\in\mathds{N}}$ are dependent, so one has to establish generalized covariance inequalities to derive moment bounds for $U_n\left(h_2\right)$.

Under the strong assumption of $\star$-mixing and the existence of 4th moments, Sen \cite{sen} showed that $\sqrt{n}U_n\left(h_2\right)\rightarrow0$ a.s.. Yoshihara \cite{yosh} weakened this to absolutely regular processes. Convergence to zero in probability of $\sqrt{n}U_n\left(h_2\right)$ was proved by Denker and Keller \cite{denk} for functionals of absolutely regular processes and by Dehling and Wendler \cite{dehl} for strongly mixing sequences. The convergence of $\sqrt{n}U_n\left(h_2\right)$ together with the Central Limit Theorem for partial sums can be used to prove the asymptotic normality of nondegenerate $U$-statistics.

In 1961, Hoeffding showed that $U_{n}\left(h_2\right)\rightarrow0$ a.s. for independent observations. If $h$ is continuous, this holds under the minimal assumption that $\left(X_n\right)_{n\in\mathds{N}}$ is ergodic, as Aaronson et. al. \cite{aaro} have proved. We give better rates of convergence for absolutely regular sequences, strongly mixing sequences and functionals of absolutely regular sequences. We will apply moment inequalities and the method of subsequences. Together with the LIL for partial sums, this will imply LIL for $U$-statistics.

For independent data, second moments of the kernel are required. For mixing data, one needs higher moments: 

\begin{definition} Let $\left(X_n\right)_{n\in\mathds{N}}$ be a stationary process. A kernel $h$ has uniform $m$-moments, if for all $k\in\mathds{N}_{0}$
\begin{align*}
\iint\left|h\left(x_{1},x_{2}\right)\right|^{m}dF\left(x_{1}\right)dF\left(x_{2}\right)&\leq M,\\
\int\left|h\left(x_{1},x_{k}\right)\right|^{m}dP\left(x_{1},x_{k}\right)&\leq M.
\end{align*}
\end{definition}

In the case of strong mixing and functionals of absolutely regular processes, one needs also a continuity condition. We consider the $P$-Lipschitz condition (see Dehling, Wendler \cite{dehl}) and the variation condition introduced by Denker and Keller \cite{denk}:
\begin{definition}
\begin{enumerate}
\item A kernel $h$ is called $P$-Lipschitz-continuous with constant $L>0$ if
\begin{equation*}
 E\left[\left|h\left(X,Y\right)-h\left(X',Y\right)\right|\mathds{1}_{\left\{\left|X-X'\right|\leq\epsilon\right\}}\right]\leq L\epsilon
\end{equation*}
for every $\epsilon>0$, every pair $X$ and $Y$ with the common distribution $\mathcal{P}_{X_{1},X_{k}}$ for a $k\in\mathds{N}$ or $\mathcal{P}_{X_{1}}\times\mathcal{P}_{X_{1}}$ and $X'$ and $Y$ also with one of these common distributions.
\item A kernel $h$ satisfies the variation condition, if there is a constant $L$ such that
\begin{equation*}
 E\left[\sup_{\left\|(x,y)-(X,Y)\right\|\leq \epsilon,\ \left\|(x',y')-(X,Y)\right\|\leq \epsilon}\left|h\left(x,y\right)-h\left(x',y'\right)\right|\right]\leq L\epsilon,
\end{equation*}
where $X$, $Y$ have the common distribution $\mathcal{P}_{X_{1}}\times\mathcal{P}_{X_{1}}$ and $\left\|(x_1,x_2)\right\|=(x_1^2+x_2^2)^{1/2}$ denotes the Euclidean norm.
\end{enumerate}

\end{definition}

\begin{example} Let $h\left(x_1,x_2\right)=\left|x_1-x_2\right|.$ As this kernel is Lipschitz-continuous, it is clear that it satisfies the $P$-Lipschitz-condition and the variation condition.
\end{example}

\begin{example} Let $h\left(x_{1},x_{2}\right)=\int_{0}^{1}\left(\mathds{1}_{\left\{x_{1}\leq t\right\}}-t\right)\left(\mathds{1}_{\left\{x_{2}\leq t\right\}}-t\right)dt$. This kernel is uniformly bounded by 1 and $P$-Lipschitz-continuous with constant 1, as
\begin{align*}
&E\left[\left|h\left(X,Y\right)-h\left(X',Y\right)\right|\mathds{1}_{\left\{\left|X-X'\right|\leq\epsilon\right\}}\right]\\
=&E\left[\left|\int_{0}^{1}\left(\mathds{1}_{\left\{X\leq t\right\}}-\mathds{1}_{\left\{X'\leq t\right\}}\right)\left(\mathds{1}_{\left\{Y\leq t\right\}}-t\right)dt\right|\mathds{1}_{\left\{\left|X-X'\right|\leq\epsilon\right\}}\right]\\
\leq& E\left[\left|\int_{0}^{1}\left(\mathds{1}_{\left\{X\leq t\right\}}-\mathds{1}_{\left\{X'\leq t\right\}}\right)dt\right|\mathds{1}_{\left\{\left|X-X'\right|\leq\epsilon\right\}}\right]=E\left[\left|X-X'\right|\mathds{1}_{\left\{\left|X-X'\right|\leq\epsilon\right\}}\right]\leq\epsilon.
\end{align*}

\end{example}

\begin{example} Let be $t\in\R$ and $h\left(x_1,x_2\right)=\mathds{1}_{\left\{\frac{1}{2}\left(x_1+x_2\right)\leq t\right\}}$. Then
\begin{equation*}
\sup_{\substack{\left\|(x,y)-(X,Y)\right\|\leq \epsilon\\ \left\|(x',y')-(X,Y)\right\|\leq \epsilon}}\left|\mathds{1}_{\left\{\frac{1}{2}\left(x+y\right)\leq t\right\}}-\mathds{1}_{\left\{\frac{1}{2}\left(x'+y'\right)\leq t\right\}}\right|=\begin{cases}1 &\text{ if } \frac{X+Y}{2}\in\left(t-\frac{\epsilon}{\sqrt{2}},t+\frac{\epsilon}{\sqrt{2}}\right]\\0&\text{ else}\end{cases}
\end{equation*}
If $X_1$ has a bounded density, then the density $f_{\frac{1}{2}\left(X+Y\right)}$ of $\frac{1}{2}\left(X+Y\right)$ is also bounded, where $X$, $Y$ are independent random variables with the same distribution as $X_1$. Then
\begin{multline*}
 E\left[\sup_{\left\|(x,y)-(X,Y)\right\|\leq \epsilon,\ \left\|(x',y')-(X,Y)\right\|\leq \epsilon}\left|h\left(x,y\right)-h\left(x',y'\right)\right|\right]\\
\leq P\left[\frac{X+Y}{2}\in\left(t-\frac{\epsilon}{\sqrt{2}},t+\frac{\epsilon}{\sqrt{2}}\right]\right]\leq \left(\sqrt{2}\sup_{t\in\R}f_{\frac{1}{2}\left(X+Y\right)}\right)\cdot\epsilon
\end{multline*}
and $h$ satisfies the variation condition.
\end{example}

\noindent\textbf{Remark. } The two continuity conditions are close in spirit. The main difference is that one has to consider all common distributions of $X$, $Y$ for checking $P$-Lipschitz continuity (that can be difficult), but only the replacement of one of the arguments of $h$, while in the variation condition, both arguments of $h$ are replaced, but only the case that $X$ and $Y$ are independent has to be considered.  

\section{Main Results}

\begin{theorem}\label{theo1} Let $\left(X_n\right)_{n\in\mathds{N}}$ be a stationary process and $h_2$ a degenerate, centered kernel with uniform $(2+\delta)$-moments for some $\delta>0$. Let $\tau\geq0$ be such that one of the following three conditions hold:
\begin{enumerate}
 \item  $\left(X_n\right)_{n\in\mathds{N}}$ is absolutely regular and $\sum_{k=0}^{n}k\beta^{\frac{\delta}{2+\delta}}\left(k\right)=O\left(n^\tau\right)$.
\item $\left(X_n\right)_{n\in\mathds{N}}$ is strongly mixing, $E\left|X_1\right|^\gamma<\infty$ for a $\gamma>0$, $h_2$ satisfies the $P$-Lipschitz-continuity or the variation condition and $\sum_{k=0}^{n}k\alpha^{\frac{2\gamma\delta}{3\gamma\delta+\delta+5\gamma+2}}\left(k\right)=O\left(n^\tau\right)$.
\item $\left(X_n\right)_{n\in\mathds{N}}$ is a $1$-approximating functional of an absolutely regular process and $h_2$ satisfies the $P$-Lipschitz-continuity or the variation condition. For $\alpha_L=\sqrt{2\sum_{i=L}^{\infty}a_i}$: $\sum_{k=0}^{n}k\left(\beta^{\frac{\delta}{2+\delta}}\left(k\right)+\alpha^{\frac{\delta}{2+\delta}}_k\right)=O\left(n^\tau\right)$.
\end{enumerate}
Then:
\begin{equation}\label{line16}
 \frac{n^{1-\frac{\tau}{2}}}{\log^{\frac{3}{2}}n\log\log n}U_{n}\left(h_2\right)\xrightarrow{a.s.}0
\end{equation}
\end{theorem}

\noindent\textbf{Remark. }Since $\beta(k)\leq 1$, condition (1) in Theorem~1 is always satisfied
with some $\tau \in [0,2]$. In the extreme case when $\tau=2$,
the conclusion of Theorem~1 is trivial, since $U_n(h_2)\rightarrow 0$ by
the $U$-statistic ergodic theorem for absolutely regular processes,
established by Aaronson et al. \cite{aaro}. In the other extreme case 
$\tau=0$, i.e. when the series 
$\sum_{k=1}^\infty k\, \beta^{\frac{\delta}{2+\delta}} (k)$ converges,
the conclusion of Theorem~1 is close to the optimal rate which follows in
the independent case from the LIL of Dehling, Denker and Philipp \cite{deh3}.

\begin{theorem}\label{theo2} Let $\left(X_n\right)_{n\in\mathds{N}}$ be a stationary process and $h$ a centered kernel with uniform $(2+\delta)$-moments for some $\delta>0$. Let $\epsilon>0$ be such that one of the following three conditions hold:
\begin{enumerate}
 \item  $\left(X_n\right)_{n\in\mathds{N}}$ is absolutely regular and $\sum_{k=0}^{n}k\beta^{\frac{\delta}{2+\delta}}\left(k\right)=O\left(n^{1-\epsilon}\right)$.
\item $\left(X_n\right)_{n\in\mathds{N}}$ is strongly mixing, $E\left|X_1\right|^\gamma<\infty$ for a $\gamma>0$, $h_2$ satisfies the $P$-Lipschitz-continuity or the variation condition and $\sum_{k=0}^{n}k\alpha^{\frac{2\gamma\delta}{3\gamma\delta+\delta+5\gamma+2}}\left(k\right)=O\left(n^{1-\epsilon}\right)$.
\item $\left(X_n\right)_{n\in\mathds{N}}$ is a $1$-approximating functional with constants $\left(a_n\right)_{n\in\N}$ of an absolutely regular process with mixing coefficients safisfying $\beta\left(n\right)=O\left(n^{-\frac{168\delta+336}{\delta}}\right)$ and for $\alpha_L=\sqrt{2\sum_{i=L}^{\infty}a_i}$: $\sum_{k=0}^{n}k\alpha^{\frac{\delta}{2+\delta}}_k=O\left(n^{1-\epsilon}\right)$. $h_2$ satisfies the $P$-Lipschitz-continuity or the variation condition, $\left(h_1\left(X_n\right)\right)_{n\in\N}$ is $(2+\delta)$-approximating with constants $\left(b_n\right)_{n\in\N}$, such that $b_n=O\left(n^{-\frac{2\delta+7}{\delta}}\right)$.
\end{enumerate}
If additionally $\sigma_\infty^2:=\Var\left[h_1\left(X_0\right)\right]+2\sum_{i=1}^{\infty}\Cov\left[h_1\left(X_0\right),h_1\left(X_i\right)\right]>0$, then the LIL holds for $T_n=\sum_{1\leq i<j\leq n}h\left(X_i,X_j\right)$.
\end{theorem}

\section{An application to robust estimation}

The classical approach to estimate the location of a sequence of random variables $\left(X_n\right)_{n\in\N}$ is based on the sample mean $\bar{X}=\frac{1}{n}\sum_{i=1}^{n}X_i$, but this estimator is not robust in the sense that a single extreme value can have a big influence on $\bar{X}$. The median of $X_1,\ldots,X_n$ is robust to outliers, but has a low efficiency if the $X_n$ are standard normal. As a compromise, one can use a trimmed mean or the Hodges-Lehmann estimator
\begin{equation*}
 H_n=\operatorname{median}\left\{\frac{X_i+X_j}{2}\big|  1\leq i<j\leq n\right\}.
\end{equation*}
The Hodges-Lehmann estimator can be expressed with the generalized inverse of the empirical $U$-distribution function
\begin{equation*}
H_n=U^{-1}_n\left(\frac{1}{2}\right):=\inf \left\{t\in\R \big|U_n\left(t\right)\geq \frac{1}{2} \right\},
\end{equation*}
with
\begin{equation*}
 U_n\left(t\right):=\frac{2}{n(n-1)}\sum_{1\leq i<j\leq n}\mathds{1}_{\left\{\frac{1}{2}\left(X_i+X_j\right)\leq t\right\}}.
\end{equation*}
Let $U\left(t\right)=E\left[\mathds{1}_{\left\{\frac{1}{2}\left(X+Y\right)\leq t\right\}}\right]$, where $X$ and $Y$ are independent. If $U\left(t\right)$ is strictly increasing and continuous, we can without loss of generality assume that $U\left(t\right)=t$ for $t\in\left[0,1\right]$. For functionals of absolutely regular processes, Borovkova, Burton and Dehling \cite{boro} have proved the convergence of the emperical $U$-process
\begin{equation*}
 \left(\sqrt{n}\left(U_n\left(t\right)-U\left(t\right)\right)\right)_{t\in\R}
\end{equation*}
to a Gaussian process. By Theorem 1 of Vervaat \cite{verv}, the same holds for the inverse process $(\sqrt{n}(U_n^{-1}(t)-EU_n^{-1}(t)))_t$, so $H_n$ is asymptotically normal. Our aim is to prove the LIL for $H_n$. First note that $H_n$ is smaller than $t_0=U^{-1}\left(\frac{1}{2}\right)$, iff $U_n\left(t_0\right)$ is bigger than $\frac{1}{2}$. This converse behavior motivates a generalized Bahadur representation
\begin{equation}\label{line3}
 H_n=t_0-\frac{U_n\left(t_0\right)-U\left(t_0\right)}{U'\left(t_0\right)}+R_n,
\end{equation}
where we need to assume that $U'\left(t_0\right)>0$ (so $U\left(t\right)$ is invertible in a neighborhood and $U\left(t_0\right)=\frac{1}{2}$). The following short calculation shows that the remainder $R_n$ is related to the inverse of the empirical $U$-process centered in $\left(t_0,U_n\left(t_o\right)\right)$. We define:
\begin{equation*}
Z_n\left(x\right):=\left(U_n\left(\cdot+t_0\right)-U_n\left(t_0\right)\right)^{-1}\left(x\right)-\frac{x}{U'\left(t_0\right)},
\end{equation*}
and observe that
\begin{align*}
 Z_n\left(x\right)&=\inf\left\{s\big|U_n\left(s+t_0\right)-U_n\left(t_0\right)\leq x\right\}-\frac{x}{U'\left(t_0\right)}\\
&=\inf\left\{s\big|U_n\left(s\right)\leq x+U_n\left(t_0\right)\right\}-\frac{x}{U'\left(t_0\right)}-t_0=U_n^{-1}\left(x+U_n\left(t_0\right)\right)-\frac{x}{U'\left(t_0\right)}-t_0.
\end{align*}
Thus we finally get
\begin{equation*}
Z_n\left(U\left(t_0\right)-U_n\left(t_0\right)\right)=H_n-t_0+\frac{U_n\left(t_0\right)-U\left(t_0\right)}{U'\left(t_0\right)}=R_n.
\end{equation*}
By Theorem \ref{theo2}, $U\left(t_0\right)-U_n\left(t_0\right)=O\left(\sqrt{\frac{\log\log n}{n}}\right)$ a.s., so if we can show that for any constant $C$
\begin{equation}\label{line4}
 \sup_{\left|t\right|\leq C\sqrt{\frac{\log\log n}{n}}}\left(U_n\left(t_0+t\right)-U_n\left(t_0\right)-U\left(t_0+t\right)+U\left(t_0\right)\right)=o\left(\sqrt{\frac{\log\log n}{n}}\right)\ \ \text{a.s.}
\end{equation}
then by Theorem 4 of Vervaat \cite{verv}
\begin{equation*}
 \sup_{\left|t\right|\leq C\sqrt{\frac{\log\log n}{n}}}Z_n\left(t\right)=o\left(\sqrt{\frac{\log\log n}{n}}\right)\ \ \text{a.s.}
\end{equation*}
and hence
\begin{equation*}
R_n=Z_n\left(U\left(t_0\right)-U_n\left(t_0\right)\right)=o\left(\sqrt{\frac{\log\log n}{n}}\right)\ \ \text{a.s.}
\end{equation*}
The LIL for $H_n$ follows then easily from the Bahadur representation (\ref{line3}) and the LIL for $U_n\left(t_0\right)$. We will only sketch the proof of (\ref{line4}). $U_n\left(t\right)$ and $U\left(t\right)$ are nondecreasing, so for $t_1<t<t_2$:
\begin{align*}
 \left|U_n\left(t\right)-U\left(t\right)\right|&\leq \max\left\{\left|U_n\left(t_1\right)-U\left(t\right)\right|,\left|U_n\left(t_2\right)-U\left(t\right)\right|\right\}\\
&\leq\max\left\{\left|U_n\left(t_1\right)-U\left(t_1\right)\right|,\left|U_n\left(t_2\right)-U\left(t_2\right)\right|\right\}+\left(U\left(t_2\right)-U\left(t_1\right)\right)
\end{align*}
Furthermore,  $U\left(t\right)$ is differentiable in $t_0$, so $\left(U\left(t_2\right)-U\left(t_1\right)\right)=O\left(t_2-t_1\right)$ as $t_1,t_2\rightarrow t_0$ and for every $\epsilon >0$ we can find a $K$ such that 
\begin{equation*}
 \sup_{\left|t\right|\leq C\sqrt{\frac{\log\log n}{n}}}\sqrt{\frac{n}{\log\log n}}Z_n\left(t\right)
\leq \max_{|k|\leq K}\sqrt{\frac{n}{\log\log n}}Z_n\left(\frac{kC}{K}\sqrt{\frac{\log\log n}{n}}\right)+\epsilon.
\end{equation*}
$Z_n\left(\frac{kC}{K}\sqrt{\frac{\log\log n}{n}}\right)$ is a $U$-statistic with kernel $\mathds{1}_{\left\{\frac{1}{2}\left(X_i+X_j\right)\in\left(t_0,t_0+\frac{kC}{K}\sqrt{\frac{\log\log n}{n}}\right]\right\}}$, which has decaying moments. Similar to Theorem \ref{theo2}, one can show that $Z_n\left(\frac{kC}{K}\sqrt{\frac{\log\log n}{n}}\right)=o\left(\sqrt{\frac{\log\log n}{n}}\right)$ a.s. if the mixing assumption (3) of Theorem \ref{theo2} holds.

\section{Preliminary results}

To control the moments of degenerate $U$-statistics, we need bounds for the covariance. In the following three lemmas, let $m=\max\left\{i_{\left(2\right)}-i_{\left(1\right)},i_{\left(4\right)}-i_{\left(3\right)}\right\}$, where $\left\{i_1,i_2,i_3,i_4\right\}=\left\{i_{(1)},i_{(2)}, i_{(3)},i_{(4)}\right\}$ and $i_{(1)}\leq i_{(2)}\leq i_{(3)}\leq i_{(4)}$:

\begin{lemma}[Yoshihara \cite{yosh}]\label{lem9} Let $h_2$ be a centered, degenerate kernel with uniform $(2+\delta)$-moments for a $\delta>0$.
If $\left(X_{n}\right)_{n\in\mathds{N}}$ is absolutely regular, then there is a constant $C$ such that
\begin{equation*}
\left|E\left[h_{2}\left(X_{i_{1}},X_{i_{2}}\right)h_{2}\left(X_{i_{3}},X_{i_{4}}\right)\right]\right|\leq C\beta^{\frac{\delta}{2+\delta}}\left(m\right).
\end{equation*}
\end{lemma}

\begin{lemma}\label{lem10} Let $h_2$ be a centered, degenerate kernel that satisfies the $P$-Lipschitz-continuity or the variation condition and has uniform $(2+\delta)$-moments for a $\delta>0$, $\left(X_{n}\right)_{n\in\mathds{N}}$ a stationary sequence of random variables. If there is a $\gamma>0$ with $E\left|X_{k}\right|^{\gamma}<\infty$, then there exists a constant $C$, such that the following inequality holds:
\begin{equation*}
\left|E\left[h_{2}\left(X_{i_{1}},X_{i_{2}}\right)h_{2}\left(X_{i_{3}},X_{i_{4}}\right)\right]\right|\leq C\alpha^{\frac{2\gamma\delta}{3\gamma\delta+\delta+5\gamma+2}}\left(m\right)
\end{equation*}
\end{lemma}

This lemma is due to Dehling, Wendler \cite{dehl} for $P$-Lipschitz-continuous kernels. The proof under the variation condition is very similar and hence omitted.

\begin{lemma}\label{lem11}Let $h_2$ be a centered, degenerate kernel that satisfies the $P$-Lipschitz-continuity or the variation condition and has uniform $(2+\delta)$-moments for a $\delta>0$, and $\left(X_{n}\right)_{n\in\mathds{N}}$ a $1$-approximating functional of an absolutely regular process with constants $a_{l}$. Define $\alpha_L$ as $\alpha_L=\sqrt{2\sum_{i=L}^{\infty}a_i}$ and $\beta\left(j\right)$ as the mixing coefficient of $\left(Z_n\right)$. Then:
\begin{equation*}
\left|E\left[h_{2}\left(X_{i_{1}},X_{i_{2}}\right)h_{2}\left(X_{i_{3}},X_{i_{4}}\right)\right]\right|\leq C\beta^{\frac{\delta}{2+\delta}}\left(\lfloor\frac{m}{3}\rfloor\right)+C\alpha_{\lfloor\frac{m}{3}\rfloor}^{\frac{\delta}{2+\delta}}
\end{equation*}
\end{lemma}

\begin{proof} First, let $h_2$ be $P$-Lipschitz-continuous. For simplicity, we consider only the case $O=i_1<i_2<i_3<i_4$ and $m=i_2-i_1\geq i_4-i_3$. With Corollary 2.17 of Borovkova et. al. \cite{boro}, there exist sequences $\left(X'_n\right)_{n\in\Z}$ and $\left(X''_n\right)_{n\in\Z}$ with the same distribution as $\left(X_n\right)_{n\in\Z}$, such that
\begin{enumerate}
 \item $\left(X''_n\right)_{n\in\Z}$ is independent of $\left(X_n\right)_{n\in\Z}$,
\item $P\left[\sum_{i=m}^{\infty}\left|X_i-X_i'\right|> \alpha_{\lfloor\frac{m}{3}\rfloor}\right]\leq \alpha_{\lfloor\frac{m}{3}\rfloor}+\beta\left(\lfloor\frac{m}{3}\rfloor\right)$,
\item $P\left[\sum_{i=0}^{\infty}\left|X'_{-i}-X''_{-i}\right|> \alpha_{\lfloor\frac{m}{3}\rfloor}\right]\leq \alpha_{\lfloor\frac{m}{3}\rfloor}$.
\end{enumerate}
As $h_2$ is degenerated and $X''_{i_{1}}$ and $\left(X_{i_{2}},X_{i_{3}},X_{i_{4}}\right)$ are independent, we have that
\begin{equation*}
E\left[h_{2}\left(X''_{i_{1}},X_{i_{2}}\right)h_{2}\left(X_{i_{3}},X_{i_{4}}\right)\right]=0,
\end{equation*}
so we can now write
\begin{align*}
&\left|E\left[h_{2}\left(X_{i_{1}},X_{i_{2}}\right)h_{2}\left(X_{i_{3}},X_{i_{4}}\right)\right]\right|\\
=&\left|E\left[h_{2}\left(X'_{i_{1}},X'_{i_{2}}\right)h_{2}\left(X'_{i_{3}},X'_{i_{4}}\right)-h_{2}\left(X''_{i_{1}},X_{i_{2}}\right)h_{2}\left(X_{i_{3}},X_{i_{4}}\right)\right]\right|\\
\leq& \left|E\left[h_{2}\left(X'_{i_{1}},X'_{i_{2}}\right)h_{2}\left(X'_{i_{3}},X'_{i_{4}}\right)-h_{2}\left(X''_{i_{1}},X'_{i_{2}}\right)h_{2}\left(X'_{i_{3}},X'_{i_{4}}\right)\right]\right|\\
&+\left|E\left[h_{2}\left(X''_{i_{1}},X'_{i_{2}}\right)h_{2}\left(X'_{i_{3}},X'_{i_{4}}\right)-h_{2}\left(X''_{i_{1}},X_{i_{2}}\right)h_{2}\left(X'_{i_{3}},X'_{i_{4}}\right)\right]\right|\\
&+\left|E\left[h_{2}\left(X''_{i_{1}},X_{i_{2}}\right)h_{2}\left(X'_{i_{3}},X'_{i_{4}}\right)-h_{2}\left(X''_{i_{1}},X_{i_{2}}\right)h_{2}\left(X_{i_{3}},X'_{i_{4}}\right)\right]\right|\\
&+\left|E\left[h_{2}\left(X''_{i_{1}},X_{i_{2}}\right)h_{2}\left(X_{i_{3}},X'_{i_{4}}\right)-h_{2}\left(X''_{i_{1}},X_{i_{2}}\right)h_{2}\left(X_{i_{3}},X_{i_{4}}\right)\right]\right|.
\end{align*}
In order to keep this proof short , we treat only the first of the four summands. Define
\begin{equation*}
h_{2,K}\left(x,y\right) = \left\{ \begin{array}{ll}
h_{2}\left(x,y\right) & \textrm{if $\left|h_{2}\left(x,y\right)\right|\leq\sqrt{K}$}\\
\sqrt{K} & \textrm{if $h_{2}\left(x,y\right)>\sqrt{K}$}\\
-\sqrt{K} & \textrm{if $h_{2}\left(x,y\right)<-\sqrt{K}$}
\end{array} \right.
\end{equation*}
It is clear that $h_{2,K}$ is $P$-Lipschitz-continuous, too. We get that
\begin{align*}
&\left|E\left[h_{2}\left(X'_{i_{1}},X'_{i_{2}}\right)h_{2}\left(X'_{i_{3}},X'_{i_{4}}\right)-h_{2}\left(X''_{i_{1}},X'_{i_{2}}\right)h_{2}\left(X'_{i_{3}},X'_{i_{4}}\right)\right]\right]\\
=&\left|E\left[\left(h_{2}\left(X'_{i_{1}},X'_{i_{2}}\right)-h_{2}\left(X''_{i_{1}},X'_{i_{2}}\right)\right)h_{2}\left(X'_{i_{3}},X'_{i_{4}}\right)\right]\right|\\
\leq&E\left[\left|\left(h_{2,K}\left(X'_{i_{1}},X'_{i_{2}}\right)-h_{2,K}\left(X''_{i_{1}},X'_{i_{2}}\right)\right)h_{2,K}\left(X'_{i_{3}},X'_{i_{4}}\right)\right|\mathds{1}_{\left\{\left|X'_{i_{1}}-X''_{i_{1}}\right|\leq\alpha_{\lfloor\frac{m}{3}\rfloor}\right\}}\right]\\
&+E\left[\left|\left(h_{2,K}\left(X'_{i_{1}},X'_{i_{2}}\right)-h_{2,K}\left(X''_{i_{1}},X'_{i_{2}}\right)\right) h_{2,K}\left(X'_{i_{3}},X'_{i_{4}}\right)\right|\mathds{1}_{\left\{\left|X'_{i_{1}}-X''_{i_{1}}\right|>\alpha_{\lfloor\frac{m}{3}\rfloor}\right\}}\right]\\
&+E\left[\left|h_{2,K}\left(X'_{i_{1}},X'_{i_{2}}\right)h_{2,K}\left(X'_{i_{3}},X'_{i_{4}}\right)-h_{2}\left(X'_{i_{1}},X'_{i_{2}}\right)h_{2}\left(X'_{i_{3}},X'_{i_{4}}\right)\right|\right]\\
&+E\left[\left|h_{2,K}\left(X''_{i_{1}},X'_{i_{2}}\right)h_{2,K}\left(X'_{i_{3}},X'_{i_{4}}\right)-h_{2}\left(X''_{i_{1}},X'_{i_{2}}\right)h_{2}\left(X'_{i_{3}},X'_{i_{4}}\right)\right|\right]\\
\end{align*}
Because of the $P$-Lipschitz-continuity and $\left|h_{2,K}\left(X'_{3},X'_{4}\right)\right|\leq\sqrt{K}$, the first summand is smaller than $2L\epsilon\sqrt{K}$. By property 3 of $\left(X'_n\right)_{n\in\Z}$ and $\left(X''_n\right)_{n\in\Z}$, the second term is bounded by
\begin{equation*}
P\left[\left|X''_{i_{1}}-X'_{i_{1}}\right|\geq\alpha_{\lfloor\frac{m}{3}\rfloor}\right]2K\leq 2\alpha_{\lfloor\frac{m}{3}\rfloor}K.
\end{equation*}
As $h_{2}\left(X'_{i_{1}},X'_{i_{2}}\right)h_{2}\left(X'_{i_{3}},X'_{i_{4}}\right)$ and $h_{2}\left(X''_{i_{1}},X'_{i_{2}}\right)h_{2}\left(X'_{i_{3}},X'_{i_{4}}\right)$ are random variables with $(1+\frac{\delta}{2})$-moments smaller than $M$ from the definition of the uniform $(2+\delta)$-moments, the third and the fourth summand are bounded by $\frac{M}{K^{\frac{\delta}{2}}}$.  Totally, we get
\begin{equation*}
\left|E\left[h_{2}\left(X'_{i_{1}},X'_{i_{2}}\right)h_{2}\left(X'_{i_{3}},X'_{i_{4}}\right)-h_{2}\left(X''_{i_{1}},X'_{i_{2}}\right)h_{2}\left(X'_{i_{3}},X'_{i_{4}}\right)\right]\right|\leq 2L\epsilon\sqrt{K}+2\alpha_{\lfloor\frac{m}{3}\rfloor}K+2\frac{M}{K^{\frac{\delta}{2}}}.
\end{equation*}
Setting $K=\left(\alpha_{\lfloor\frac{k}{3}\rfloor}+\beta\left(\lfloor\frac{k}{3}\rfloor\right)\right)^{-\frac{2}{2+\delta}}M^{\frac{2}{2+\delta}}$, keeping in mind that this $K$ is nondecreasing and treating the other three summands in the same way, one easily obtains
\begin{equation*}
 \left|E\left[h_{2}\left(X_{i_{1}},X_{i_{2}}\right)h_{2}\left(X_{i_{3}},X_{i_{4}}\right)\right]\right|\leq CM^{\frac{2}{2+\delta}}\left(\beta^{\frac{\delta}{2+\delta}}\left(\lfloor\frac{m}{3}\rfloor\right)+\alpha_{\lfloor\frac{m}{3}\rfloor}^{\frac{\delta}{2+\delta}}\right)
\end{equation*}
for a constant $C$, which proofs the lemma for a $P$-Lipschitz-continuous kernel. Let now $h_2$ satisfy the variation condition. Obviously, the same holds for $h_{2,K}$ and
\begin{align*}
&\left|E\left[h_{2}\left(X_{i_{1}},X_{i_{2}}\right)h_{2}\left(X_{i_{3}},X_{i_{4}}\right)\right]\right|\\
\leq& \left|E\left[h_{2}\left(X'_{i_{1}},X'_{i_{2}}\right)h_{2}\left(X'_{i_{3}},X'_{i_{4}}\right)-h_{2}\left(X''_{i_{1}},X_{i_{2}}\right)h_{2}\left(X'_{i_{3}},X'_{i_{4}}\right)\right]\right|\\
&+\left|E\left[h_{2}\left(X''_{i_{1}},X_{i_{2}}\right)h_{2}\left(X'_{i_{3}},X'_{i_{4}}\right)-h_{2}\left(X''_{i_{1}},X_{i_{2}}\right)h_{2}\left(X_{i_{3}},X_{i_{4}}\right)\right]\right|.
\end{align*}
Again, we concentrate on the first summand. By the variation condition, we have that
\begin{multline*}
E\left[\left|\left(h_{2,K}\left(X'_{i_{1}},X'_{i_{2}}\right)-h_{2,K}\left(X''_{i_{1}},X_{i_{2}}\right)\right)h_{2,K}\left(X'_{i_{3}},X'_{i_{4}}\right)\right|\mathds{1}_{\left\{\left|X'_{i_{1}}-X''_{i_{1}}\right|\leq\alpha_{\lfloor\frac{m}{3}\rfloor},\left|X'_{i_{2}}-X_{i_{2}}\right|\leq\alpha_{\lfloor\frac{m}{3}\rfloor}\right\}}\right]\\
\leq \sqrt{K} E\left[\sup_{\left\|(x,y)-(X''_{i_1},X_{i_2})\right\|\leq \sqrt{2}\alpha_{\lfloor\frac{m}{3}\rfloor},\ \left\|(x',y')-(X''_{i_1},X_{i_2})\right\|\leq \sqrt{2}\alpha_{\lfloor\frac{m}{3}\rfloor}}\left|h_{2,K}\left(x,y\right)-h_{2,K}\left(x',y'\right)\right|\right]\\
\leq 2\sqrt{2}\sqrt{K}L\epsilon.
\end{multline*}

As $P\left[\left|X_{i_1}''-X_{i_{1}}\right|\geq\alpha_{\lfloor\frac{m}{3}\rfloor}\right]\leq\alpha_{\lfloor\frac{m}{3}\rfloor}$, $P\left[\left|X_{i_2}'-X_{i_{2}}\right|\geq\alpha_{\lfloor\frac{m}{3}\rfloor}\right]\leq\alpha_{\lfloor\frac{m}{3}\rfloor}+\beta\left(\lfloor\frac{m}{3}\rfloor\right)$, it follows that
\begin{multline*}
E\left[\left|\left(h_{2,K}\left(X'_{i_{1}},X'_{i_{2}}\right)-h_{2,K}\left(X''_{i_{1}},X_{i_{2}}\right)\right)h_{2,K}\left(X'_{i_{3}},X'_{i_{4}}\right)\right|\mathds{1}_{\left\{\left|X'_{i_{1}}-X''_{i_{1}}\right|>\alpha_{\lfloor\frac{m}{3}\rfloor},\left|X'_{i_{2}}-X_{i_{2}}\right|>\alpha_{\lfloor\frac{m}{3}\rfloor}\right\}}\right]\\
\leq P\left[\left|X''_{i_{1}}-X'_{i_{1}}\right|\geq\alpha_{\lfloor\frac{m}{3}\rfloor},\left|X_{i_{2}}-X'_{i_{2}}\right|\geq\alpha_{\lfloor\frac{m}{3}\rfloor}\right]2K\leq 4\left(\alpha_{\lfloor\frac{m}{3}\rfloor}+\beta\left(\lfloor\frac{m}{3}\rfloor\right)\right)K.
\end{multline*}
The rest of the proof is the same as above.

\end{proof}

Yoshihara \cite{yosh} deduced the following moment bound under condition (1) with the help of Lemma \ref{lem9}. The result follows from condition (2) and (3) in the same way using the Lemmas \ref{lem10} and \ref{lem11} instead.

\begin{lemma}\label{lem12}
Let $\left(X_n\right)_{n\in\mathds{N}}$ be a stationary process and $h_2$ a degenerate, centered kernel with uniform $(2+\delta)$-moments for a $\delta>0$. Let be $\tau\geq0$ such that one of the following three conditions hold:
\begin{enumerate}
 \item  $\left(X_n\right)_{n\in\mathds{N}}$ is absolutely regular and $\sum_{k=0}^{n}k\beta^{\frac{\delta}{2+\delta}}\left(k\right)=O\left(n^\tau\right)$.
\item $\left(X_n\right)_{n\in\mathds{N}}$ is strongly mixing, $E\left|X_1\right|^\gamma<\infty$ for a $\gamma>0$, $h_2$ satisfies the $P$-Lipschitz-continuity or the variation condition and $\sum_{k=0}^{n}k\alpha^{\frac{2\gamma\delta}{\gamma\delta+\delta+5\gamma+2}}\left(k\right)=O\left(n^\tau\right)$.
\item $\left(X_n\right)_{n\in\mathds{N}}$ is a $1$-approximating functional of an absolutely regular process and $h_2$ satisfies the $P$-Lipschitz-continuity or the variation condition. For $\alpha_L=\sqrt{2\sum_{i=L}^{\infty}a_i}$: $\sum_{k=0}^{n}k\left(\beta^{\frac{\delta}{2+\delta}}\left(k\right)+\alpha^{\frac{\delta}{2+\delta}}_k\right)=O\left(n^\tau\right)$.
\end{enumerate}
Then
\begin{equation*}
\sum_{i_{1},i_{2},i_{3},i_{4}=1}^{n}\left|E\left[h_{2}\left(X_{i_{1}},X_{i_{2}}\right)h_{2}\left(X_{i_{3}},X_{i_{4}}\right.)\right]\right|=O\left(n^{2+\tau}\right).
\end{equation*}
\end{lemma}

\begin{lemma}\label{lem13}
 If $h$ satisfies the $P$-Lipschitz-continuity or the variation condition, then the condition holds also for $h_2$.
\end{lemma}

\begin{proof} For $P$-Lipschitz-continuous kernels, we refer to Dehling, Wendler \cite{dehl}, proof of Lemma 3.3. Let now $h$ satisfy the variation condition. As $h_2\left(x,y\right)=h\left(x,y\right)-h_1\left(x\right)-h_1\left(y\right)-\theta$, it suffices to verify this condition for $h_1$. Recall that $h_1\left(x\right)=E\left[h\left(x,Y\right)\right]-\theta$, so
\begin{align*}
&E\left[\sup_{\left\|(x,y)-(X,Y)\right\|\leq \epsilon,\ \left\|(x',y')-(X,Y)\right\|\leq \epsilon}\left|h_1\left(x\right)-h_1\left(x'\right)\right|\right]\\
=&E\left[\sup_{\left|x-X\right|\leq \epsilon,\ \left|x'-X\right|\leq \epsilon}\left|E\left[h\left(x,Y\right)\right]-E\left[h\left(x',Y\right)\right]\right|\right]\displaybreak[0]\\
\leq&E\left[\sup_{\left|x-X\right|\leq \epsilon,\ \left|x'-X\right|\leq \epsilon}E\left|h\left(x,Y\right)-h\left(x',Y\right)\right|\right]\displaybreak[0]\\
\leq&E\left[\sup_{\left|x-X\right|\leq \epsilon,\ \left|x'-X\right|\leq \epsilon}\left|h\left(x,Y\right)-h\left(x',Y\right)\right|\right]\\
\leq&E\left[\sup_{\left\|(x,y)-(X,Y)\right\|\leq \epsilon,\ \left\|(x',y')-(X,Y)\right\|\leq \epsilon}\left|h\left(x,y\right)-h\left(x',y'\right)\right|\right]\leq L\epsilon.
\end{align*}

\end{proof}

\section{Proofs of the theorems.}

\begin{proof}[Proof of Theorem \ref{theo1}]:
We define
\begin{align*}
 Q_n&=\sum_{1\leq i_1<i_2\leq n}h_2\left(X_{i_1},X_{i_2}\right)\\
a_n&=\frac{1}{n^{1+\frac{\tau}{2}}\log^{\frac{3}{2}}n\log\log n}.
\end{align*}
With the method of subsequences, it suffices to show that
\begin{align}\label{line29}
a_{2^l}Q_{2^l}\left(h_2\right)&\xrightarrow{\text{a.s.}}0\\
\label{line30}\max_{2^{l-1}\leq n<2^l}\left|a_nQ_n-a_{2^{l-1}}Q_{2^{l-1}}\right|&\xrightarrow{\text{a.s.}}0
\end{align}
as $l\rightarrow\infty$. We use the Chebyshev inequality and Lemma \ref{lem12} to prove the first line. For every $\epsilon>0$:
\begin{equation*}
\sum_{l=1}^{\infty}P\left[\left|a_{2^l}Q_{2^l}\left(h_2\right)\right|>\epsilon\right]\leq\frac{1}{\epsilon^2}\sum_{l=1}^{\infty}a_{2^l}^2E\left[Q_{2^l}^2\left(h_2\right)\right]\leq C\frac{1}{\epsilon^2}\sum_{l=1}^{\infty}\frac{1}{l^{\frac{3}{2}}\log l}<\infty
\end{equation*}
(\ref{line29}) follows with the Borel-Cantelli Lemma. To prove (\ref{line30}), we first have to find a bound for the second moments, using a well known chaining technique. For example, by the triangle inequality we have
\begin{equation*}
\left|a_{15}Q_{15}-a_8Q_8\right|\leq\left|a_{15}Q_{15}-a_{14}Q_{14}\right|+\left|a_{14}Q_{14}-a_{12}Q_{12}\right|+\left|a_{12}Q_{12}-a_8Q_8\right|.
\end{equation*}
Using such a decomposition for all $n$ with $2^{l-1}\leq n<2^l$, we conclude that
\begin{multline*}
\max_{2^{l-1}\leq n<2^l}\left|a_nQ_n-a_{2^{l-1}}Q_{2^{l-1}}\right|\\
\leq \sum_{d=1}^{l} \max_{i=1,\ldots,2^{l-d}}\left|a_{2^{l-1}+i2^{d-1}}Q_{2^{l-1}+i2^{d-1}}-a_{2^{l-1}+(i-1)2^{d-1}}Q_{2^{l-1}+(i-1)2^{d-1}}\right|.
\end{multline*}
As for any random variables $Y_1,\ldots,Y_n$: $E\left(\max\left|Y_i\right|\right)^2\leq \sum EY_i^2$, it follows that

\begin{multline*}
\quad E\left[\left(\max_{2^{l-1}\leq n<2^l}\left|a_nQ_n-a_{2^{l-1}}Q_{2^{l-1}}\right|\right)^2\right]\\
\shoveleft\leq l\sum_{d=1}^{l}\sum_{i=1}^{2^{l-d}}E\left[\left(a_{2^{l-1}+i2^{d-1}}Q_{2^{l-1}+i2^{d-1}}-a_{2^{l-1}+(i-1)2^{d-1}}Q_{2^{l-1}+(i-1)2^{d-1}}\right)^2\right]\\
\shoveleft\leq l\sum_{d=1}^{l}\sum_{i=1}^{2^{l-d}}E\left[\left(a_{2^{l-1}+i2^{d-1}}\left(Q_{2^{l-1}+i2^{d-1}}-Q_{2^{l-1}+(i-1)2^{d-1}}\right)\right.\right.\\
\left.\left.+\left(a_{2^{l-1}+i2^{d-1}}-a_{2^{l-1}+(i-1)2^{d-1}}\right)Q_{2^{l-1}+(i-1)2^{d-1}}\right)^2\right]\\
\shoveleft\leq l\sum_{d=1}^{l}\sum_{i=1}^{2^{l-d}}2a_{2^{l-1}+i2^{d-1}}^2E\left[\left(Q_{2^{l-1}+i2^{d-1}}-Q_{2^{l-1}+(i-1)2^{d-1}}\right)^2\right]\\
+l\sum_{d=1}^{l}\sum_{i=1}^{2^{l-d}}2\left(a_{2^{l-1}+i2^{d-1}}-a_{2^{l-1}+(i-1)2^{d-1}}\right)^2E\left[Q^2_{2^{l-1}+(i-1)2^{d-1}}\right]\\
\\
\shoveleft=\sum_{d=1}^{l}2a_{2^{l-1}+i2^{d-1}}^2E\left[\sum_{i=1}^{2^{l-d}}\left(Q_{2^{l-1}+i2^{d-1}}-Q_{2^{l-1}+(i-1)2^{d-1}}\right)^2\right]\\
+l\sum_{d=1}^{l}\sum_{i=1}^{2^{l-d}}2\left(a_{2^{l-1}+i2^{d-1}}+a_{2^{l-1}+(i-1)2^{d-1}}\right)\left(a_{2^{l-1}+i2^{d-1}}-a_{2^{l-1}+(i-1)2^{d-1}}\right)E\left[Q^2_{2^{l-1}+(i-1)2^{d-1}}\right]\\
\leq l^26a^2_{2^{l-1}}\sum_{i_{1},i_{2},i_{3},i_{4}=1}^{2^l}\left|E\left[h_{2}\left(X_{i_{1}},X_{i_{2}}\right)h_{2}\left(X_{i_{3}},X_{i_{4}}\right)\right]\right|\leq C\frac{1}{l\log^2 l}.
\end{multline*}
In the last line we used the fact that the sequence $\left(|a|_{n}\right)_{n\in\mathds{N}}$ is decreasing and Lemma \ref{lem12}. It now follows for all $\epsilon>0$ with the Chebyshev inequality
\begin{multline*}
\sum_{l=1}^{\infty}P\left[\max_{2^{l-1}\leq n<2^l}\left|a_nQ_n-a_{2^{l-1}}Q_{2^{l-1}}\right|>\epsilon\right]\leq \frac{1}{\epsilon^2}\sum_{l=1}^{\infty}E\left[\left(\max_{2^{l-1}\leq n<2^l}\left|a_nQ_n-a_{2^{l-1}}Q_{2^{l-1}}\right|\right)^2\right]\\
\leq\frac{C}{\epsilon^2}\sum_{l=1}^{\infty}\frac{1}{l\log^2 l}<\infty,
\end{multline*}
the Borel-Cantelli Lemma completes the proof.

\end{proof}

\begin{proof}[Proof of Theorem \ref{theo2}]:
We give the proof the theorem only under condition (1) and omit the similar proofs under conditions (2) and (3) (where Lemma \ref{lem13} is used to conclude that $h_2$ is $P$-Lipschitz-continuous and Rio's result \cite{rio} has to be replaced by the result of Philipp and Stout \cite[chapter 7]{phi2} under condition (3)).

First note that $E\left|h_1\left(X_n\right)\right|^{2+\delta}\leq E\left|h\left(X,Y\right)\right|^{2+\delta}<\infty$ ($X$, $Y$ being independent). By standard arguments
\begin{equation*}
 \frac{1}{n}\Var\left[\sum_{i=1}^{n}h_1\left(X_i\right)\right]\xrightarrow{n\rightarrow\infty}\sigma_\infty^2=\Var\left[h_1\left(X_0\right)\right]+2\sum_{i=1}^{\infty}\Cov\left[h_1\left(X_0\right),h_1\left(X_i\right)\right] 
\end{equation*}
an by Lemma \ref{lem12} $\Var\left[U_n\left(h_2\right)\right]\rightarrow0$. So we have
\begin{align}
 \Var\left[\sum_{1\leq i<j\leq n}h\left(X_i,X_j\right)\right]&\xrightarrow{n\rightarrow\infty}\infty\\
\frac{\Var\left[\sum_{1\leq i<j\leq n}h\left(X_i,X_j\right)\right]}{\Var\left[(n-1)\sum_{i=1}^{n}h_1\left(X_i\right)\right]}&\xrightarrow{n\rightarrow\infty}1.\label{line1}
\end{align}

As $\left(\beta(n)\right)_{n\in\N}$ is nonincreasing and $\sum_{k=0}^{n}k\beta^{\frac{\delta}{2+\delta}}\left(k\right)=O\left(n^{1-\epsilon}\right)$, it follows that
\begin{equation*}
 \beta\left(n\right)^{\frac{\delta}{2+\delta}}=O\left(n^{-(1+\epsilon)}\right)\ \Rightarrow\ \sum_{k=1}^{n}k^{\frac{1}{1+\delta}}\alpha\left(k\right)\leq\sum_{k=1}^{n}k^{\frac{1}{1+\delta}}\beta\left(k\right)<\infty,
\end{equation*}
so by Theorem 2 of Rio \cite{rio} the LIL holds for $\sum_{i=1}^{n}h_1\left(X_i\right)$. By Theorem 1 and Line \ref{line1}, this holds also for $\sum_{1\leq i<j\leq n}h\left(X_i,X_j\right)$.
\end{proof}

\small{

}\end{document}